\documentclass[12pt]{amsproc}
\usepackage{amsmath,amsthm}
\usepackage{amsfonts,amssymb}
\newtheorem{thm}{Theorem}[section]

\newtheorem{lem}[thm]{Lemma}
\newtheorem{prop}[thm]{Proposition}

\theoremstyle{remark}

\newcommand{\be}{\begin{equation}}
\newcommand{\ee}{\end{equation}}
\newcommand{\bea}{\begin{eqnarray}}

\newcommand{\eea}{\end{eqnarray}}
\newcommand{\Bea}{\begin{eqnarray*}}
\newcommand{\Eea}{\end{eqnarray*}}

\def\CL{{\mathcal L}}

\def\C{{\mathbb C}}
\def\H{{\mathbb H}}
\def\N{{\mathbb N}}
\def\R{{\mathbb R}}

\def\Z{{\mathbb Z}}

\def\1{\text{\bf {1}}}

\begin{document}

\title[$L^p$ boundedness for the solution of wave equation]
{ $L^p$ estimates for the wave equation associated to the Grushin operator}
\author{K. Jotsaroop and S. Thangavelu }

\address{Department of Mathematics\\ Indian Institute
of Science\\Bangalore-560 012}
\email{jyoti@math.iisc.ernet.in, veluma@math.iisc.ernet.in}

\date{\today}
\keywords{}
\subjclass{42C, 42C10, 43A90, 46E30, 46E40, 47D09}
\thanks{ }

\begin{abstract}
We prove that the solution of the wave equation associated to the
Grushin operator $ G = -\Delta-|x|^2 \partial_t^2 $ is bounded on
$ L^p(\R^{n+1}) (1<p<\infty),$ when $|\frac{1}{p}-\frac{1}{2}|<\frac{1}{n+2}.$


\end{abstract}

\maketitle

\section{Introduction}
\setcounter{equation}{0}

Consider the solution of the initial value problem
$$ \partial_t^2 u(x,t)
= \Delta u(x,t), u(x,0) = 0, \partial_t u(x,0) = f(x) $$
for the standard
wave equation associated to the Laplacian on $ \R^n.$ Representing the
solution as $ u(x,t) = \frac{\sin t\sqrt{-\Delta}}{\sqrt{-\Delta}}f(x) $ one
can investigate the $ L^p $ mapping properties of
$\frac{\sin t\sqrt{-\Delta}}{\sqrt{-\Delta}}.$ This problem has been studied by several authors: Peral \cite{JP} and Miyachi \cite{AM} have obtained  sharp
range of $p,$ viz.  $|\frac{1}{p}-\frac{1}{2}|\leq\frac{1}{n-1},$ for which
$ \frac{\sin t\sqrt{-\Delta}}{\sqrt{-\Delta}} $ is bounded on $ L^p(\R^n).$
Other $ L^p-L^q $ estimates were considered e.g. by Strichartz \cite{RS}. The
case of the Hermite operator $ -\Delta+|x|^2 $ has been treated by one of us
\cite{NT} and more general operators of the form
$ -\Delta+V $ by Zhong \cite{JZ}. In all these cases optimal range of $ p $
for which the solution operator is bounded on $ L^p(\R^n)$ is known.

All the operators mentioned above are elliptic but results for operators from
the subelliptic case are also available. The wave equation associated to the
sublaplacian $ \CL $ on the Heisenberg group $ \H^n $ has been studied by Mueller and Stein \cite{MS}. They have shown that the solution operator
$\frac{\sin t\sqrt{\CL}}{\sqrt{\CL}}$ is bounded on $ L^p(\H^n) $ for all
$p$ satisfying $|\frac{1}{p}-\frac{1}{2}|<\frac{1}{d-1} $ where $ d =2n+1 $
is the Euclidean dimension of $ \H^n.$ The interesting
point to note here is the appearance of the $ d-1 $ rather that $ Q-1 $ where
$ Q = 2n+2 $ is the homogeneous dimension. The weaker result with $ Q-1 $ in place of $ d-1$ is known from earlier works. Also when we consider only
functions on the Heisenberg group which are band limited in the central
variable the range can be further extended to
$|\frac{1}{p}-\frac{1}{2}|<\frac{1}{2n-1} $ as was shown in \cite{NS}.

In this article we are interested in the wave equation associated to the
Grushin operator $  G = -\Delta-|x|^2 \partial_t^2 $
on $ \R^{n+1}.$ Though this operator is very similar to the sublaplacian with
a very explicit spectral decomposition study of spectral multipliers poses
formidable problems due to a lack of group structure on $ \R^{n+1} $
compatible with the operator. However, $ G $ can be obtained from $ \CL $
on the Heisenberg group via a certain representation and hence in principle
transference techniques can be used to prove weaker versions of multiplier
theorems. As the dimension of $ \H^n $ is $ (2n+1) $ whereas $ G $ lies on an
$ n+1 $ dimensional space results obtained via transference are far from
optimal. In a recent work \cite{JST} the authors have studied multipliers
associated to $ G $ using operator valued Fourier multipliers.

The study of wave equation associated to the Grushin operator in one
dimension has been initiated by Ralf Meyer \cite{RM}. In his unpublished
thesis written
under the guidance of Detlef Mueller, he has proved the following theorem.
He considers the class of functions which are supported in
$S_{\mathcal{C}_{1}}=\{(x,t):|x|\leq \mathcal{C}_{1}\}.$

\begin{thm} For every $\mathcal{C}_{1},$ $s>0,$ $1\leq p\leq\infty$ and
$\alpha > |\frac{1}{p}-\frac{1}{2}|$
there exists a constant $ C = C_{p,s,\mathcal{C}_{1}}^{\alpha}$
such that for all $f\in L^p(\R^2)$ with support contained  in
$S_{\mathcal{C}_{1}}$ the estimates
$$\|\frac{\cos s\sqrt{G}}{(1+G)^{\alpha/2}}f\|_{L^p(\R^2)}\leq C
 \|f\|_{L^p(\R^2)},$$
and
$$\|\frac{\sin s\sqrt{G}}{\sqrt{G}(1+G)^{(\alpha-1)/2}}f\|_{L^p(\R^2)}\leq
C \|f\|_{L^p(\R^2)}$$
are valid.
\end{thm}

Meyer has proved the above theorem by following the approach used by Mueller
and Stein for the Heisenberg group. By a very careful analysis of certain 
kernels he obtained
extremely  delicate estimates which were possible only under some
assumptions on the support. It is almost impossible either to get rid of
this assumption or to use the same method in higher dimensions. Fortunately,
there is an alternate approach which we have used elsewhere in studying
multipliers for the Grushin operator. That approach allows us to prove
$ L^p $ estimates for the wave equation associated to higher dimensional
Grushin operators.

The idea is to consider multipliers for $ G $ as operator valued multipliers
for the one dimensional Euclidean Fourier transform. To elaborate on this let
us consider the spectral decomposition of $ G.$ Let
$$ f^\lambda(x) = \int_{-\infty}^\infty f(x,t) e^{i\lambda t} dt $$ stand for
the inverse Fourier transform of $ f(x,t) $ in the $ t $ variable. Then by applying $ G $ to the inversion formula
$$ f(x,t) = \frac{1}{2\pi} \int_{-\infty}^\infty e^{-i\lambda t}
f^\lambda(x) d\lambda $$ we see that
$$ Gf(x,t) = \frac{1}{2\pi} \int_{-\infty}^\infty e^{-i\lambda t} H(\lambda)
f^\lambda(x) d\lambda $$
where $ H(\lambda) = -\Delta+\lambda^2 |x|^2 $ is the scaled Hermite operator
on $ \R^n.$ The spectral decomposition $  H(\lambda) $ is explicitly known
and given by
$$  H(\lambda) = \sum_{k=0}^\infty (2k+n)|\lambda| P_k(\lambda) $$ where
$ P_k(\lambda)$ are the Hermite projections, see \cite{ST}. (We will say more
about these projections later in the paper). Consequently the spectral
decomposition of $ G $ is written as
$$ Gf(x,t) = \frac{1}{2\pi}\int_{-\infty}^\infty  e^{-i\lambda t} \left(
\sum_{k=0}^\infty (2k+n)|\lambda| P_k(\lambda)f^\lambda(x)\right)
d\lambda .$$

Given a bounded function $ m $ on the spectrum of $ G $ which is just the
half line $ [0,\infty) $ we can define $ m(G) $ by spectral theorem. In view
of the above decomposition we see that
$$ m(G)f(x,t) = \frac{1}{2\pi} \int_{-\infty}^\infty e^{-i\lambda t}
m(H(\lambda))f^\lambda(x) d\lambda $$
where the Hermite multiplier $ m(H(\lambda))$ is given by
$$ m(H(\lambda)) = \sum_{k=0}^\infty m((2k+n)|\lambda|) P_k(\lambda).$$ Let
$ X = L^p(\R^n) $ and identify $ L^p(\R^{n+1}) $ with $ L^p(\R,X),$ the
$ L^p $ space of Banach space valued functions on $ \R.$ With this
identification we see that $ m(G) $ can be considered as a Fourier multiplier
on $ \R $ for $ X $ valued functions, the multiplier being given by
$ m(H(\lambda))$. Of course, we need to assume that $ m(H(\lambda))$ are
uniformly bounded on $ X = L^p(\R^n) $ even for the boundedness of $ m(G) $
on $ L^2(\R,X).$ Further conditions are needed to guarantee the boundedness
of $ m(G) $ on $ L^p(\R,X).$

Fortunately for us the problem of operator valued multipliers has been studied
by L.Weis \cite{LW} and he has obtained some sufficient conditions. The
following theorem has been proved in slightly more general set up. Given
a function $ m$ taking values in $ B(X,Y),$ the space of bounded linear
operators from $ X $ into $ Y $ one can define
$$ T_mf(t) = \frac{1}{2\pi}\int_{-\infty}^\infty e^{i\lambda t} m(\lambda)
\hat{f}(\lambda) d\lambda $$
for all $ f \in L^2(\R,X).$ This operator is clearly bounded from $L^2(\R,X)$
into $L^2(\R,Y)$ provided $ m(\lambda) $ uniformly bounded. For such operators
we have the following result.

\begin{thm} Let  $ X $ and $ Y $ be UMD spaces. Let
$ m:\R^* \rightarrow B(X,Y) $ be a differentiable function such that the
families $ \{ m(\lambda): \lambda \in \R^* \} $ and
$ \{ \lambda \frac{d}{d\lambda}m(\lambda): \lambda \in \R^* \} $
are R-bounded. Then $ m $ defines a Fourier multiplier which is bounded
from $ L^p(\R,X) $ into $ L^p(\R,Y) $ for all $ 1 < p < \infty.$
\end{thm}

Note that mere uniform boundedness of $ m(\lambda) $ and
$ \lambda \frac{d}{d\lambda}m(\lambda) $ are not enough to guarantee the
$ L^p $ boundedness of the Fourier multiplier. As the reader may recall they
are sufficient in the scalar case. In most applications of the above theorem,
the crux of the matter lies in proving the
R-boundedness of these families. For our main result we only need to use this
theorem when $ X = Y = L^p(\R^n) $ in which case the R-boundedeness is
equivalent to  vector-valued inequalities for $ m(\lambda) $
and $ \lambda \frac{d}{d\lambda}m(\lambda).$ Indeed, the R-boundedness of
a family of operators $ T(\lambda) $ is equivalent to the inequality
$$ \|\left(\sum_{j=1}^\infty |T(\lambda_j)f_j|^2
\right)^{\frac{1}{2}}\|_p
\leq C \|\left(\sum_{j=1}^\infty |f_j|^2 \right)^{\frac{1}{2}}\|_p $$
for all possible choices of $ \lambda_j \in \R^* $ and
$ f_j \in L^p(\R^n) .$ Thus we only need to verify this vector-valued
inequality for the two families in the theorem.

We consider the following initial value problem for the wave equation:
\be{\partial_s^2u(x,t;s)+Gu(x,t;s)=0 }
\ee $$u(x,t;0)=0, \partial_su(x,t;0)=f(x,t).$$
By using the functional calculus for $G,$ it is easy to see that
the solution of the above equation is given by
$$u(x,t;s)=\frac{\sin s\sqrt{G}}{\sqrt{G}}f(x,t) .$$
Since $G$ is a homogeneous operator of degree $(n+2)$ under the nonisotroic
dilation $D_sf(x,t)=f(sx,s^2t),$ it is enough to consider the case $s=1.$
Our main result is the following theorem.

\begin{thm} Let $ n \geq 2.$ The operator
$\frac{\sin \sqrt{G}}{\sqrt{G}}$ is bounded on $L^p(\R^{n+1})$ for
all $ p $ satisfying $|\frac{1}{p}-\frac{1}{2}|<\frac{1}{n+2}.$
\end{thm}

Note that in the above theorem the homogeneous dimension $ (n+2) $ occurs. We
believe that the optimal result is the one in which $ (n+2) $ can be replaced
by $n.$ The Fourier multiplier corresponding to
$ \frac{\sin s\sqrt{G}}{\sqrt{G}} $ is given by
$ \frac{\sin s\sqrt{H(\lambda)}}{\sqrt{H(\lambda)}} $ which is precisely the
solution operator for the wave equation associated to the Hermite operator.
For a fixed $ \lambda $ the boundedness of this operator on $ L^p(\R^n) $ is
known for the range $|\frac{1}{p}-\frac{1}{2}|<\frac{1}{n},$ see \cite{NT},
\cite{JZ}. What we need to prove is the R-boundedness of the above family as
well as the same for $ \lambda $ times its derivative. The major part of this
paper is concerned with this problem.

For proving Theorem 1.3 we consider a more general class of
oscillatory multipliers of $G$, viz.
$\frac{J_{\alpha}(\sqrt{G})}{\sqrt{G}^{\alpha}}, $ for $ \Re{\alpha}\geq -1/2
$ where $ J_{\alpha}$ is the Bessel function of order $\alpha.$
This is a densely defined analytic family of operators acting on
$L^p(\R^{n+1}).$ When $\alpha=1/2$ we get back the solution operator of
the wave  equation and hence Theorem 1.3 follows once we  prove

\begin{thm} Let $ n \geq 2 $ and $ 1 < p < \infty.$  Then
$\frac{J_{\alpha}(\sqrt{G})}{\sqrt{G}^{\alpha}}$ is bounded on
$L^p(\R^{n+1})$ whenever $ \Re{(\alpha)}>
(n+2)|\frac{1}{p}-\frac{1}{2}|-\frac{1}{2}.$
\end{thm}

Recall that the Bessel functions $ J_\alpha(t) $ are defined even for complex
values of $\alpha.$ In fact the Poisson integral representation
$$ J_\alpha(t) = \frac{(t/2)^\alpha}{\Gamma((2\alpha+1)/2)\Gamma(1/2)}
\int_{-1}^1 e^{its}(1-s^2)^{(2\alpha-1)/2} ds $$
is valid as long as $ \Re{(\alpha)} > -1/2.$ Moreover, when
$ \alpha = \beta+\delta+i\gamma $ where
$ \beta >-1/2, \delta > 0, \gamma \in \R $ we have the identity
$$ \frac{J_{\alpha}(t)}{t^\alpha} = \frac{2^{1-\delta-i\gamma}}{\Gamma(\delta+i\gamma)}
\int_0^1 \frac{J_{\beta}(st)}{(st)^\beta}
(1-s^2)^{\delta+i\gamma-1}s^{2\beta+1} ds.$$
Thus we see that  $\frac{J_{\alpha}(\sqrt{G})}{\sqrt{G}^{\alpha}}$ ia an
analytic family of
operators which is bounded on $ L^2(\R^{n+1}) $ whenever $ \Re{(\alpha)}
= -\frac{1}{2} .$ Using the above formula we can also check that the family
is admissible. Hence we can appeal to Stein's analytic interpolation theorem
to obtain Theorem 1.4 as soon as we get

\begin{thm} Let $ n \geq 2.$ Then
$\frac{J_{\alpha}(\sqrt{G})}{\sqrt{G}^{\alpha}}$ is bounded on
$L^p(\R^{n+1})$ for all $ 1< p < \infty $ provided $ \Re{(\alpha)}>
\frac{(n+1)}{2}.$
\end{thm}

Thus by setting $ m_\alpha(u) = \frac{J_\alpha(\sqrt{u})}{\sqrt{u}^\alpha} $
we study the R-boundedness of the family $ m_\alpha(H(\lambda))$ when
$ \Re{(\alpha)}> \frac{(n+1)}{2}.$ We also need to study the R-boundedness of
$ \lambda \frac{d}{d\lambda}m_\alpha(H(\lambda)).$ We address these problems
in the next two sections.

We conclude this introduction with the following remarks. In all the theorems
stated above we have assumed $ n \geq 2.$ The reason is the following: in
the proof of Proposition 2.2 which is used in proving Theorem 2.1 we need to
use the estimate $ \Phi_k(x,x) \leq C (2k+n)^{n/2-1}, x \in \R^n $
which is valid only when $ n \geq 2.$ Here $ \Phi_k(x,y) $ is the kernel
of the projection $ P_k $ associated to the Hermite operator $ H.$ In the
one dimensional case we have $ \Phi_k(x,x) = (h_k(x))^2 ,$ where $ h_k $ is
the $k-$th Hermite function on $ \R $ behaves like $ k^{-1/6} $ and hence we
do not get an analogue of Proposition 2.2. However, when $ B $ is a compact
subset of $ \R $ we do have $ \sup_{x \in B}(h_k(x))^2 \leq C (2k+1)^{-1/2}$
and hence it is possible to prove a version of Theorem 1.3 for the operator
$ \chi_B \frac{\sin{\sqrt{G}}}{\sqrt{G}}\chi_B .$  We do not pursue this here
as the result of Meyer is stronger than what we can prove.

\section{ A maximal theorem for $ m_\alpha(H(\lambda))$}
\setcounter{equation}{0}

As we mentioned at the end of the introduction we are interested in proving
vector valued inequalities for the families $ T_\alpha(\lambda) =
 m_\alpha(H(\lambda)) $ and $  \lambda \frac{d}{d\lambda}T_\alpha(\lambda).$
In order to do that we need a maximal theorem for the family
$ T_\alpha(\lambda) $ which means that we have to get estimates for the
maximal function $ T_\alpha^* f(x) =  \sup_{\lambda\in\R^*}
|T_\alpha(\lambda)f(x)|.$ For $ 1 \leq p < \infty $ let $ M_pf(x) =
(M|f|^p(x))^{1/p} $ where $ Mf $ is the Hardy-Littlewood maximal function.
Let $\alpha=x+iy.$ We call $c(\alpha)$ an admissible function (or function
of admissible growth) if $$\sup_{y\in\R}e^{-b|y|}\log(|c(\alpha)|)<\infty$$
for some $b<\pi.$ With this terminolgy  we have the following:

\begin{thm} Let $ n \geq 2.$ (i) For $ \Re(\alpha) > \frac{(n-1)}{2},$
we have 
$$ T_\alpha^* f(x)
\leq C_2(\alpha)M_2f(x);$$ and (ii) for $ \Re(\alpha) > n-\frac{1}{2},  
T_\alpha^* f(x) \leq C_1(\alpha) Mf(x) $ where the functions $ C_1 $ and 
$ C_2 $ are of admissible growth.
\end{thm}

This theorem will be proved by obtaining good estimates on the kernel of
$ T_\alpha(\lambda).$ We briefly recall some details from the spectral theory
of the Hermite operator $ H(\lambda).$ Let $ \Phi_\alpha, \alpha \in \N^n $
stand for the normalised Hermite functions on $ \R^n $ which are
eigenfunctions of $ H = H(1) $ with eigenvalues $ (2|\alpha|+n) $ and form an
orthonormal basis for $ L^2(\R^n).$ It follows that for $ \lambda \in \R^* $
the functions $ \Phi_\alpha^\lambda(x) = |\lambda|^{n/4}
\Phi(|\lambda|^{1/2}x) $ satisfy $ H(\lambda)\Phi_\alpha^\lambda
= (2|\alpha|+n)|\lambda| \Phi_\alpha^\lambda.$
The spectral projections $ P_k(\lambda) $ of $ H(\lambda) $ are defined by
$$ P_k(\lambda)f =  \sum_{|\alpha| =k}
(f,\Phi_\alpha^\lambda)\Phi_\alpha^\lambda.$$ It follows that
$  P_k(\lambda) = \delta_\lambda  P_k \delta_\lambda^{-1} $ where $
\delta_\lambda f(x) = f(|\lambda|^{\frac{1}{2}}x) $ and $ P_k =P_k(1).$
Therefore, $ m(H(\lambda)) = \delta_\lambda m(\lambda H) \delta_\lambda^{-1}$
for any multiplier $m.$

The above remarks imply that
$$ \frac{J_{\alpha}(\sqrt{H(\lambda)})}{\sqrt{H(\lambda)}^{\alpha}}f(x) =
\delta_{\lambda}\frac{J_{\alpha}(\sqrt{|\lambda|H})}{(\sqrt{|\lambda|H})^
{\alpha}}{\delta_{\lambda}}^{-1}f(x).$$ In view of this relation, a moment's
thought reveals that it is enough to consider the maximal function
$$ \sup_{t>0}|\frac{J_{\alpha}(t\sqrt{H})}{(t\sqrt{H})^{\alpha}}f(x)| $$
and establish the estimates stated in the theorem above.

By the definition
$$ \frac{J_{\alpha}(t\sqrt{H})}{(t\sqrt{H})^{\alpha}}f =
\sum_{k=0}^\infty \frac{J_{\alpha}(t\sqrt{2k+n})}{({t\sqrt{2k+n})}^{\alpha}}
P_k f $$
and hence it follows that
$ \frac{J_{\alpha}(t\sqrt{H})}{(t\sqrt{H})^{\alpha}}$ is an integral operator
whose kernel $ K_t^\alpha(x,y) $ is given by
$$ K_t^{\alpha}(x,y) = \sum_{k=0}^\infty \frac{J_{\alpha}(t\sqrt{2k+n})}
{({t\sqrt{2k+n})}^{\alpha}}\Phi_k(x,y),$$
where $\Phi_k(x,y)=\sum_{|\beta|=k}\Phi_{\beta}(x)\Phi_{\beta}(y) $ is the
kernel of $ P_k.$ We require the following estimates on the kernel
$ K_t^\alpha.$

\begin{prop} Let $ n \geq 2.$ (i)  For $ \Re{(\alpha)}>\frac{n-1}{2}$ we have
$$ \int_{|x-y|>r}|K_t^{\alpha}(x,y)|^2dy\leq C_2(\alpha) t^{-n}
(1+rt^{-1})^{-2\Re{(\alpha)}-1} $$ and (ii) for
$ \Re{(\alpha)}>n-\frac{1}{2}$ we have
$$ \sup_{|x-y|>r}|K_t^{\alpha}(x,y)|\leq C_1(\alpha) t^{-n}
(1+rt^{-1})^{-\Re{(\alpha)}-1/2} $$
where $C_1$ and $ C_2$ are functions  of admissible growth.
\end{prop}

Assuming the proposition for a moment, we complete the proof of Theorem 2.1.
For $ x \in \R^n $ we define
$ f_k(y)=\chi_{\{y:2^{k}<|x-y|\leq2^{k+1}\}}(y)f(y),$ $k\in\Z $ so that
$f = \sum_{k=-\infty}^\infty f_k $
and
$$ \frac{J_{\alpha}(t\sqrt{H})}{(t\sqrt{H})^{\alpha}}f(x)=
\sum_{k=-\infty}^\infty \int_{\R^n}K_t^{\alpha}(x,y)f_k(y)dy.$$
After applying Cauchy-Schwarz inequality to each term in the sum we see that
$ \frac{J_{\alpha}(t\sqrt{H})}{(t\sqrt{H})^{\alpha}}f(x)$ is bounded by
$$  \sum_{k=-\infty}^\infty 2^{(k+1)n/2}
\left( \int_{|x-y|>2^k}|K_t^{\alpha}(x,y)|^2 dy \right)^{\frac{1}{2}}
\left(\frac{1}{2^{(k+1)n}}\int_{|x-y|\leq 2^{k+1}}|f(y)|^2 dy
\right)^{\frac{1}{2}}.$$
As the second factor inside the summation is bounded by $ M_2f(x), $ in view
of Proposition 2.2 we have, whenever $ \Re(\alpha) > \frac{n-1}{2}$, the
estimate
$$ |\frac{J_{\alpha}(t\sqrt{H})}{(t\sqrt{H})^{\alpha}}f(x)|\leq $$
$$ C(\alpha)\left(\sum_{k=-\infty}^\infty (2^kt^{-1})^{n/2}
(1+2^k t^{-1})^{-\Re{\alpha}-1/2}\right) M_2f(x).$$
Thus we are left with proving that
$$ G(t):=\sum_{k=-\infty}^\infty (2^kt^{-1})^{n/2}(1+2^k t^{-1})^{-\Re{\alpha}-1/2}$$ is a uniformly bounded function  of $ t>0.$
Note that $G(2^{i}t)=G(t)$, for all $ i\in\Z $ and hence it is enough to
prove the boundedness of $G$ on the interval $[1,2].$ But  $G$ is a continuous function
on $[1,2]$ as the series converges uniformly on this interval
when $\Re{\alpha}>\frac{n-1}{2}.$

This proves part (i) of Theorem 2.1. To prove the second part we proceed as
above and use the second estimate of Proposition 2.2 which is valid when
$ \Re{(\alpha)} > n-\frac{1}{2}.$ The details are left to the reader.

We now turn our attention to the proof of Proposition 2.2.  We will treat the
cases $ rt^{-1} \leq 1 $ and $ rt^{-1} > 1 $ separately. In the former case
we only need to show that
$$ \int_{|x-y|>r}|K_t^{\alpha}(x,y)|^2dy\leq C(\alpha) t^{-n}. $$
Since
$$ \int_{|x-y|>r}|K_t^{\alpha}(x,y)|^2dy\leq
\int_{\R^n}|K_t^{\alpha}(x,y)|^2dy,$$
we will actually estimate the second integral in the above inequality.
Recalling the definition of $ K_t^\alpha(x,y) $ and using the orthogonality
of the Hermite functions we see that
$$ \int_{\R^n} |K_t^{\alpha}(x,y)|^2dy = \sum_{k=0}^\infty
 \bigg{|}\frac{J_{\alpha}(t\sqrt{2k+n})}{({t\sqrt{2k+n})}^{\alpha}}\bigg{|}^2
\Phi_k(x,x).$$
Splitting the sum into two parts we first consider the sum
$$ \sum_{j = 1}^{\infty} \sum_{2^{-j} < t \sqrt{2k+n}\leq 2^{-j+1}}\bigg{|}
\frac{J_{\alpha}(t\sqrt{2k+n})}{({t\sqrt{2k+n})}^{\alpha}}\bigg{|}^2
\Phi_k(x,x).$$
As $|\frac{J_{\alpha}(s)}{s^{\alpha}}|\leq c(\alpha)$ \cite{GW}, where $c(\alpha)$
is an admissible function of $\alpha,$ for all
$ s \geq 0 $ the above sum is bounded by
$$ c(\alpha)^2 \sum_{j = 1}^{\infty} \sum_{2^{-j} < t \sqrt{2k+n}\leq 2^{-j+1}}
\Phi_k(x,x).$$
Finally, we make use of the estimate $ \Phi_k(x,x) \leq C (2k+n)^{n/2-1} $ proved in \cite{ST}
(see Lemma 3.2.2, Chapter 3  )  valid for $ n \geq 2 $ to see that the above sum
is bounded by
$$ \sum_{j = 1}^{\infty} (2^{-2j}t^{-2})^{n/2} \leq C t^{-n} $$
which takes care of the first sum.

To estimate the second sum, namely
$$ \sum_{j=0}^\infty \sum_{t\sqrt{2k+n}\sim2^j}\bigg{|}
\frac{J_{\alpha}(t\sqrt{2k+n})}{({t\sqrt{2k+n})}^{\alpha}}\bigg{|}^2
\Phi_k(x,x) $$
we make use of the estimate
$$ \bigg{|}\frac{J_{\alpha}(t\sqrt{2k+n})}{({t\sqrt{2k+n})}^{\alpha}}\bigg{|}
\leq c(\alpha) (t\sqrt{2k+n})^{-\Re{(\alpha)}-1/2},$$ when $t\sqrt{2k+n}\geq 1.$ As before this leads to
the estimate
$$ c(\alpha)^2 \sum_{j=0}^\infty \sum_{2^j < t \sqrt{2k+n}\leq 2^{j+1}}
(t\sqrt{2k+n})^{-2\Re{(\alpha)}-1}(2k+n)^{\frac{n}{2}-1}.$$
On simplifying this sum further we get the estimate
$$c(\alpha)^2 t^{-n} \sum_{j=0}^\infty 2^{-2j(\Re{(\alpha)}+\frac{1-n}{2})}.$$
The sum over $j$ converges if and only if $\Re{\alpha}>\frac{n-1}{2}.$ This
takes care of the second sum. Thus we have proved the required estimate
when $ rt^{-1} \leq 1.$

We now treat the second case namely when $ rt^{-1} > 1.$ We estimate the
integral when $ \Re{(\alpha)} > \frac{n-1}{2} $ first. Note that it is
enough to prove the estimate
$$ \int_{|x-y|>r}|K_t^{\alpha}(x,y)|^2dy\leq c(\alpha)t^{-n+2m}r^{-2m}$$
for some integer $ m > \Re{(\alpha)}+\frac{1}{2}.$
Since
$$ \int_{|x-y|>r}|K_t^{\alpha}(x,y)|^2dy\leq r^{-2m}\int ||x-y|^m
K_t^{\alpha}(x,y)|^2dy, $$
it is enough to prove
$$ \int |(x-y)^{\beta} K_t^{\alpha}(x,y)|^2dy\leq c(\alpha)t^{-n+2m}$$
for all $\beta\in\N^n$ with $|\beta|=m.$
In order to do this we make use of Lemma $3.2.3$ in  \cite{ST}, which we state
below for the convenience of the reader.

Given a function $ \psi $ defined on $ [0,\infty) $  consider the kernel
$ M_\psi $ defined by
$$ M_\psi(x,y) = \sum_{\mu \in \N^n} \psi(|\mu|) \Phi_\mu(x)\Phi_\mu(y).$$
Let $ \Delta \psi(s) = \psi(s+1)-\psi(s) $ be the forward finite difference
and let $ \Delta^k\psi $ be defined inductively. Let $ \Delta^kM_\psi $ stand
for the kernel $ M_{\Delta^k\psi}.$ We also define $ B_j=-\partial_{y_j}+y_j,
$ and $ A_j=-\partial_{x_j}+x_j$ for $j=1,2,.......n.$ For multi-indices
$ \mu, A^\mu, B^\mu $ are defined in the usual manner. With these notations
we have

\begin{lem} For any multi-index $ \beta \in \N^n $ we have
$$ (x-y)^\beta M_\psi(x,y) = \sum_{\gamma,\mu}C_{\gamma, \mu}
(B-A)^{\gamma}\triangle^{|\mu|} M_{\psi}(x,y),$$
where the sum is extended over all multi-indices $ \mu $ and $ \gamma $
satisfying $ 2\mu_j-\gamma_j = \beta_j, \mu_j\leq \beta_j.$
\end{lem}

Let us fix $ \beta \in \N^n $ with $ |\beta| = m .$ In view of the above
lemma
$$(x-y)^{\beta} K_t^{\alpha}(x,y)= \sum_{\gamma,k}c_{\gamma k}
(B-A)^{\gamma}\triangle^k M_{\psi}(x,y),$$
where $ \psi(|\mu|)=
\frac{J_{\alpha}(t\sqrt{2|\mu|+n})}{({t\sqrt{2|\mu|+n})}^{\alpha}} $
and the sum is extended over all $ \gamma $ and $ k $ with
$|\gamma|=2k-m, k\leq m .$
On expanding $(B-A)^{\gamma}$ the above becomes a finite linear combination of
terms of the following form:
$$\sum_{\mu}\triangle^k\psi(|\mu|)A^{\tau}\Phi_{\mu}(x)B^{\sigma}
\Phi_{\mu}(y),$$
where $|\tau|+|\sigma|=|\gamma|.$ By mean value theorem we can write
$$\triangle^k\psi(|\mu|)=\int_{0}^{1}....\int_{0}^{1}
\psi^{(k)}(|\mu|+s_1 +.......+ s_k)ds_1 ds_2....ds_k,$$
and hence  it is enough to prove that
$$ \int_{\R^n} |\sum_{\mu}\psi^{(k)}(|\mu|)A^{\tau}\Phi_{\mu}(x)B^{\sigma}
\Phi_{\mu}(y)|^2 dy \leq C t^{-n+2m} $$
for each $ \tau, \sigma $ and $ k $ as above.

We make use of the facts
$$ A_j\Phi_{\mu}(x)=(2|\mu_j|+2)^{\frac{1}{2}}\Phi_{\mu+ e_j}(x),~~~
 B_j\Phi_{\mu}(y)=(2|\mu_j|+2)^{\frac{1}{2}}\Phi_{\mu+ e_j}(y) $$
(see \cite{ST}) where $ e_j $ are the co-ordinate vectors. In view of this
the above integral is dominated by
$$\sum_{N = 0}^\infty |\psi^{(k)}(N)|^2(2N+n)^{|\tau|+|\sigma|}
\Phi_{N+|\tau|}(x,x).$$
Again, if we use the estimate $\Phi_N(x,x)\leq C (2N+n)^{\frac{n}{2}-1}$ and
the fact that $|\tau|+|\sigma|= 2k-m$ the above is dominated by
$$\sum_{N =0}^\infty |\psi^k(N)|^2(2N+n)^{2k-m +\frac{n}{2}-1}.$$
Now recall that $\psi(N)=
\frac{J_{\alpha}(t\sqrt{2N+n})}{({t\sqrt{2N+n})}^{\alpha}}$, so that
$\psi^{(k)}(N)=\frac{d^k}{d\lambda^k}
\frac{J_{\alpha}(t\sqrt{\lambda})}{({t\sqrt{\lambda})}^{\alpha}}
\large|_{\lambda=2N+n}.$
By making use of the well known relation
$$ \frac{d}{d\lambda}\frac{J_{\alpha}(\sqrt{\lambda})}
{({\sqrt{\lambda})}^{\alpha}}
=-\frac{1}{2}\frac{J_{\alpha+1}(\sqrt{\lambda})}
{({\sqrt{\lambda})}^{\alpha+1}},$$ (see \cite{GW})
we get
$$ \psi^{(k)}(N)= t^{2k} \frac{J_{\alpha+k}(t\sqrt{2N+n})}
{({t\sqrt{2N+n})}^{\alpha+k}}.$$ Plugging this in
the above expression we get
$$\int |\sum_{\mu}\psi^{(k)}(|\mu|)A^{\tau}\Phi_{\mu}(x)B^{\sigma}
\Phi_{\mu}(y)|^2 dy $$
$$ \leq C \sum_{N=0}^\infty \bigg|t^{2k}\frac{J_{\alpha+k}(t\sqrt{2N+n})}
{({t\sqrt{2N+n})}^{\alpha+k}}\bigg|^2(2N+n)^{2k-m +\frac{n}{2}-1}.$$

As before we estimate the above sum by splitting it into two parts. For the
part
$$ \sum_{j =1}^\infty \sum_{t\sqrt{2N+n}\sim 2^{-j}}
\bigg|t^{2k}\frac{J_{\alpha+k}(t\sqrt{2N+n})}{({t\sqrt{2N+n})}^{\alpha+k}}
\bigg|^2(2N+n)^{2k-m +\frac{n}{2}-1}$$
we use the boundedness of the Bessel function which results in the estimate
$$ c_k(\alpha)^2 \sum_{j =1}^\infty t^{4k}
\left(\frac{2^{-2j}}{t^2}\right)^{2k-m +\frac{n}{2}}$$
$$ = c_k(\alpha)^2 t^{-n+2m}\sum_{j = 1}^\infty 2^{-2j(2k-m)}2^{-nj}.$$
Since $2k-m=|\gamma|\geq 0$ the above sum clearly converges.
To treat the sum
$$ \sum_{j =0}^\infty \sum_{t\sqrt{2N+n}\sim 2^j}\bigg|t^{2k}\frac{J_{\alpha+k}(t\sqrt{2N+n})}{({t\sqrt{2N+n})}^{\alpha+k}}\bigg|^2
(2N+n)^{2k-m +\frac{n}{2}-1} $$
we make use the estimate
$$ |\frac{J_{\alpha+k}(t\sqrt{2N+n})}{({t\sqrt{2N+n})}^{\alpha+k}}|\leq
c_k(\alpha)(t\sqrt{2N+n})^{-\Re{(\alpha+k)}-\frac{1}{2}}.$$
Using the above estimate and simplifying  we get
$$ c_k(\alpha)^2\sum_{j =0}^\infty t^{4k}2^{-2j(\Re{(\alpha+k)}+\frac{1}{2})}
\left(\frac{2^{2j}}{t^2}\right)^{2k-m +\frac{n}{2}} $$
$$=c_k(\alpha)^2 t^{-n+2m}\sum_{j=0}^\infty
2^{-2j(\Re{(\alpha)}+\frac{1-n}{2}))}2^{2j(k-m)}.$$
As $k\leq m,$ the sum over $j$ is finite as soon as
$\Re{(\alpha)}>\frac{n-1}{2}.$

Thus Proposition 2.2 (i) is completely proved when $ \Re{(\alpha)} >
\frac{n-1}{2}.$ What remains to be considered is the second part for $ \Re{(\alpha)}>
n-1/2. $ Here also we consider two cases, when  $rt^{-1}\leq 1$ and $ rt^{-1} > 1.$
When $rt^{-1}\leq 1,$ it is enough to show that
$$\sup_{|x-y|>r}|K_t^{\alpha}(x,y)|\leq c(\alpha)t^{-n},$$ for $\Re{\alpha}>n-1/2.$
Clearly, $$\sup_{|x-y|>r}|K_t^{\alpha}(x,y)|\leq \sup_{x,y\in\R^n}|K_t^{\alpha}(x,y)|.$$
So it suffices to show that $$\sup_{x,y\in\R^n}|K_t^{\alpha}(x,y)|\leq c(\alpha)t^{-n}.$$
Using the definition of $K_t^{\alpha}$ we get that
$$|K_t^{\alpha}(x,y)|\leq \sum_{k}|\frac{J_{\alpha}(t\sqrt{2k+n})}{({t\sqrt{2k+n})}^{\alpha}}||\Phi_k(x,y)|.$$
Since $\Phi_k(x,y)=\sum_{|\beta|=k}\Phi_{\beta}(x)\Phi_{\beta}(y),$
an application of Cauchy-Schwarz inequality gives us
$$|\sum_{|\beta|=k}\Phi_{\beta}(x)\Phi_{\beta}(y)|\leq \sqrt{\Phi_k(x,x)\Phi_k(y,y)}.$$
For $n\geq 2$ we know that \cite{ST} $\sup_{x\in\R^n}\Phi_k(x,x)\leq
(2k+n)^{\frac{n}{2}-1}.$
Using this estimate we get that
$$|K_t^{\alpha}(x,y)|\leq \sum_{k}|\frac{J_{\alpha}(t\sqrt{2k+n})}{({t\sqrt{2k+n})}^{\alpha}}|(2k+n)^{\frac{n}{2}-1}.$$
Now, proceeding as in the previous part i.e. splitting the sum into two
 parts and using the estimates of the Bessel function we get the desired
inequality for $\Re{(\alpha)}>n-1/2.$

When $rt^{-1}>1,$ it is enough to show that
$$\sup_{|x-y|>r}|K_t^{\alpha}(x,y)|\leq c(\alpha)t^{-n+\Re{(\alpha)}+1/2}
r^{-\Re{(\alpha)}-1/2} $$
for $\Re{\alpha}>n-\frac{1}{2}.$ As before we only need to show that
$$\sup_{|x-y|>r}|K_t^{\alpha}(x,y)|\leq c(\alpha) t^{-n+m}r^{-m} $$
for some $ m > \Re{(\alpha)}+1/2 $ which in turn will follow once we
show that
$$ \sum_{|\beta| =m} \sup_{x,y \in \R^n}|(x-y)^{\beta}K_t^{\alpha}(x,y)|\leq
c(\alpha)t^{-n+m}.$$
Keeping the notation same as in the previous part, the above estimate will
follow from the estimates
$$\sup_{x,y\in\R^n}|\sum_{\mu}\psi^{(k)}(|\mu|)A^{\tau}\Phi_{\mu}(x)
B^{\sigma}\Phi_{\mu}(y)|\leq c(\alpha)t^{-n+m},$$
where $ k \leq m $ and $ |\tau|+|\sigma| = 2k-m \geq 0.$

Recalling the action of $ A_j $ and $ B_j $ on Hermite functions we see that
$$ \sum_{|\mu| = N} |A^{\tau}\Phi_{\mu}(x) B^{\sigma}\Phi_{\mu}(y)| \leq
C (2N+n)^{\frac{1}{2}(|\tau|+|\sigma|)}\sqrt{\Phi_{N+m}(x,x)\Phi_{N+m}(y,y)}.$$
Using the fact that $ |\tau|+|\sigma| = 2k-m $ the estimates on
$ \Phi_k(x,x) $ leads to
$$ \sum_{|\mu| = N} |A^{\tau}\Phi_{\mu}(x) B^{\sigma}\Phi_{\mu}(y)| \leq
C (2N+n)^{k-m/2+\frac{n}{2}-1}.$$
Recalling that
$ \psi^{(k)}(N)=t^{2k}\frac{J_{\alpha+k}(t\sqrt{2N+n})}
{({t\sqrt{2N+n})}^{\alpha+k}}$
we need to estimate
$$\sum_{N =0}^\infty \bigg|t^{2k}\frac{J_{\alpha+k}(t\sqrt{2N+n})}
{({t\sqrt{2N+n})}^{\alpha+k}}\bigg| (2N+n)^{k-\frac{m}{2} +\frac{n}{2}-1}.$$
As before splitting the above sum into two parts and using the estimates on
Bessel function we get the required estimate for $\Re{(\alpha)}>n-\frac{1}{2}.$
Thus Proposition 2.2 is completely proved.

In the next section when we try to prove the R-boundedness of $ \lambda
\frac{d}{d\lambda}T_\alpha(\lambda) $ we encounter the family $ H(\lambda)
m_{\alpha+1}(H(\lambda)).$ Hence we require the following maximal theorem
for this family.

\begin{thm} Let $ n \geq 2.$ (i) For $ \Re(\alpha) > \frac{(n+1)}{2}$,
$$ \sup_{\lambda \in \R^*}|H(\lambda)m_{\alpha+1}(H(\lambda))f(x)|
\leq C_2(\alpha)M_2f(x);$$
\newline (ii) For $ \Re(\alpha) > n+\frac{1}{2}$,
$$ \sup_{\lambda \in \R^*}|H(\lambda)m_{\alpha+1}(H(\lambda))f(x)|
\leq C_1(\alpha) Mf(x) $$
where the functions $ C_1 $ and $ C_2 $ are of admissible growth.
\end{thm}

In order to prove this theorem we need an analogue of Proposition 2.2 for
the kernel
$$ \tilde{K_t}^\alpha(x,y) = \sum_{k=0}^\infty t^2(2k+n)
\frac{J_{\alpha+1}(t\sqrt{2k+n})}{(t\sqrt{2k+n})^{\alpha+1}}\Phi_k(x,y).$$
This kernel is estimated just like the kernel $ K_t^\alpha.$ Note that
when $ t\sqrt{2k+n} \leq 1 $ both
$ \frac{J_{\alpha}(t\sqrt{2k+n})}{(t\sqrt{2k+n})^{\alpha}} $ and $
t^2 (2k+n) \frac{J_{\alpha+1}(t\sqrt{2k+n})}{(t\sqrt{2k+n})^{\alpha+1}}$ are
bounded. On the other hand when  $ t\sqrt{2k+n} \geq 1 $
$$ t^2 (2k+n) |\frac{J_{\alpha+1}(t\sqrt{2k+n})}{(t\sqrt{2k+n})^{\alpha+1}}|
\leq C(\alpha)(t\sqrt{2k+n})^{-\Re{(\alpha)}+1/2} $$
and since we are assuming $ \Re{(\alpha)} > \frac{n+1}{2} $ the same
estimates as in Proposition are satisfied by $ \tilde{K_t}^\alpha(x,y).$
This takes care of the part when $rt^{-1}\leq 1.$
Recall that in the proof of Proposition 2.2 when $rt^{-1}> 1$ we need some estimates on the derivative
of the multiplier.
The $ k-$th derivative of the function
$ t^2u m_{\alpha+1}(t^2u)$ at $ |\mu| $ is given by
$$ t^{2k}\left[k\frac{J_{\alpha+k}(t\sqrt{2|\mu|+n})}
{({t\sqrt{2|\mu|+n})}^{\alpha+k}}+ t^2(2|\mu|+n)
\frac{J_{\alpha+1+k}(t\sqrt{2|\mu|+n})}{({t\sqrt{2|\mu|+n})}^{\alpha+1+k}}
\right]$$ which can be estimated in a similar way as in the case of $K_t^{\alpha}.$
We leave the details to the reader.

\section{The R-boundedness of $ T_\alpha(\lambda)$ and $\frac{d}{d\lambda}
T_{\alpha}(\lambda)$}
\setcounter{equation}{0}

\subsection{The R-boundedness of $ m_\alpha(H(\lambda))$}

Making use of the maximal theorem proved in the previous section we will now
prove the required vector valued inequalities for the family
$ T_\alpha(\lambda) = m_\alpha(H(\lambda)).$ Using the result of Proposition 
2.2 it is possible to get the estimate
$$ \left(\int_{|x-y|>r}|K_t^{\alpha}(x,y)|^pdy \right)^{1/p} 
\leq C_2(\alpha) t^{-n/(2p')}
(1+rt^{-1})^{-\Re{(\alpha)}-1/2+n(\frac{1}{p}-\frac{1}{2})} $$
for $ 1 < p \leq 2 $ for $ \Re{(\alpha)}>\frac{n-1}{2}.$ This will lead as 
before to the estimate
$$ \sup_{\lambda \in \R^*} |T_\alpha(\lambda)f(x)| \leq C(\alpha)M_pf(x) $$
whenever $ p \geq 2.$ Unfortunately, this estimate is not good enough to yield 
the required vector-valued inequality for the family $ T_\alpha(\lambda).$ 
What we can prove is the inequality
$$ \| \left(\sum_{j=1}^\infty |T_\alpha(\lambda_j)f_j|^r\right)^{1/r}\|_p 
\leq C \| \left(\sum_{j=1}^\infty |f_j|^r\right)^{1/r}\|_p $$
for all $ r > p \geq 2.$  As we need the case $ r = 2 $ we have to 
proceed in a 
different way using analytic interpolation. We first prove the following:

\begin{prop} For any  $\varphi\in C_c^{\infty}(\R^n) $ the operator $T_{\alpha}(\lambda)$
satisfies the following
$$ \int_{\R^n}|T_{\alpha}(\lambda)f(x)|^2|\varphi(x)| dx\leq c(\alpha)
\int_{\R^n} |f(x)|^2 M\varphi(x)dx$$
 for $\Re{(\alpha)}>\frac{n-1}{2}.$ Moreover, $c(\alpha)$ is an admissible function of
 $\alpha$ and is independent of the choice of $\varphi$ and $\lambda.$
\end{prop}
\begin{proof} We make use of Theorem 2.1 along with a lemma due to Fefferman
and Stein (see Lemma 1, sec. 3 in \cite{FS}) which states that
$$ \int_{\R^n}  Mf(x)^r|\varphi(x)|dx \leq C_r \int |f(x)|^r M\varphi(x)dx,$$
for any $ r > 1 $ with $ C_r $ independent of $ f $ and $\varphi.$ Therefore,
for $ \Re{(\alpha)} > \frac{n-1}{2} $ we have, from Theorem 2.1,
$$ \int_{\R^n} |T_\alpha(\lambda)f(x)|^{p} |\varphi(x)|dx
\leq C_2(\alpha) \int (M|f|^2(x))^{p/2} |\varphi(x)| dx $$
which upon using Fefferman-Stein Lemma  yields
$$ \int_{\R^n} |T_\alpha(\lambda)f(x)|^{p} |\varphi(x)|dx
\leq C_3(\alpha) \int |f(x)|^{p} M\varphi(x) dx $$
for any $ p > 2.$ Similarly, when $ \Re{(\alpha)}>n-\frac{1}{2}$ we have
$$  \int_{\R^n} |T_\alpha(\lambda)f(x)|^{p} |\varphi(x)|dx
\leq C_4(\alpha)  \int_{\R^n} |f(x)|^{p} M\varphi(x) dx  $$
for all $ p > 1.$ Thus we see that
$$ T_{\alpha}(\lambda): L^p(\R^n, (M\varphi)~~dx)\longrightarrow
L^p(\R^n, |\varphi|~~dx)$$
is bounded for all $p > 2 $ if $\Re{(\alpha)}>\frac{n-1}{2}$
and for all $ p > 1 $ if  $\Re{(\alpha)}>n-\frac{1}{2}.$

We want to apply Stein's anaytic interpolation theorem \cite{ES}(see Chapter 5, Theorem 4.1)
 to the family
$ T_{\alpha}(\lambda).$ It is easy to see that the norm of
$T_{\alpha}(\lambda)$ is independent of $\varphi$ and $\lambda$ in both
cases and is an admissible family of operators in $\alpha.$
Fix $\alpha\in\C$ such that $\Re{\alpha}>\frac{n-1}{2}$ and $\lambda\neq0.$
Let $\delta>0$ be chosen so that  $\Re{\alpha}=\frac{n-1}{2}+\delta.$
Define an analytic family of operators $ S_z $ on the strip
$ S = \{z\in\C: 0 \leq\Re{(z)}\leq 1\}$
by setting
$S_z f=T_{(nz+n-1+\delta)/2}(\lambda).$  Let $ \epsilon = \frac{1}{2}(
-1+\sqrt{1+(8\delta)/n}) $ and take $ p_0 = 2+\epsilon $ and
$ p_1 = 1+\epsilon.$ Then it is clear, from Theorem 2.1, that
$$ \int_{\R^n}|S_{iy}f(x)|^{p_0}|\varphi|~~~dx\leq C_1(iy)\int_{\R^n}|f(x)|^{p_0}M\varphi~~~ dx $$
and
$$\int_{\R^n}|S_{1+iy}f(x)|^{p_1}|\varphi|~~~dx\leq C_2(1+iy)\int_{\R^n}|f(x)|^{p_1}M\varphi~~~dx, $$ where $C_1(iy)$
and $C_2(1+iy)$ are admissible functions and are independent of $\varphi$ and $\lambda.$
By interpolation, it follows that $ S_{\delta/n} $ is bounded from
$ L^{p}(\R^n,(M\varphi)~~dx) $ into $ L^{p}(\R^n, |\varphi|~~ dx) $ where
$ \frac{1}{p} = \frac{1-\delta/n}{2+\epsilon}+\frac{\delta/n}{1+\epsilon}.$ A
simple calculation recalling the definition of $ \epsilon,$ shows that
$ p =2 $ and hence $ S_{\delta/n} = T_{\frac{n-1}{2}+\delta}(\lambda) $ is
bounded from $ L^{2}(\R^n,(M\varphi)~~dx) $ into
$ L^{2}(\R^n, |\varphi|~~ dx) $ which proves the theorem when
$ \alpha $ is real.

When $ \alpha $ is not real we write $ \alpha = \beta+\delta+i\gamma $ where
$ \beta > \frac{n-1}{2}, \delta > 0 $ and make use of the identity
$$ \frac{J_{\beta+\delta+i\gamma}(\sqrt{H(\lambda)})}{(\sqrt{H(\lambda)})^{
\beta+\delta+i\gamma}} = \frac{2^{1-\delta-i\gamma}}{\Gamma(\delta+i\gamma)}
\int_0^1 \frac{J_{\beta}(s\sqrt{H(\lambda)})}{(s\sqrt{H(\lambda)})^\beta}
(1-s^2)^{\delta+i\gamma-1}s^{2\beta+1} ds.$$

\end{proof}

We can now prove the vector valued inequality for
$\{T_{\alpha}(\lambda)\}_{\lambda\in\R^*}$ thus proving the R-boundedness of
$ m(H(\lambda)).$

\begin{thm} Let $T_{\alpha}(\lambda)$ be as defined before. Then for any
choice of $ \lambda _j \in \R^* $ and $ f_j \in L^p(\R^n) $ we have
$$\|\left(\sum_{j=1}^{\infty}|T_{\alpha}(\lambda_j)f_j|^2
\right)^{\frac{1}{2}}\|_p
\leq C \|\left(\sum_{j=1}^{\infty}|f_j|^2\right)^{\frac{1}{2}}\|_p$$
for all $1<p<\infty$ provided  $\Re{(\alpha)}>\frac{n-1}{2}.$
\end{thm}
\begin{proof} When $p=2$ the vector valued inequality follows trivially as
$\frac{J_{\alpha}(\sqrt{(2k+n)|\lambda_j|})}{({\sqrt{(2k+n)|\lambda_j|})}^{\alpha}}$ is uniformly bounded
independent of $j$ and $k$ for any  $\alpha$ with $\Re{\alpha}
\geq -\frac{1}{2}$.
So it follows that $\|T_{\alpha}(\lambda_j)f_j\|_2\leq C\|f_j\|_2,$
where $C$ is independent of $j $ and hence we get the vector valued
inequality for $p=2.$

We will now deal with the case $p\neq2$. Without loss of generality we can
assume that $p>2$ as the case $1<p<2$ can be treated  using a duality
argument. Let $\frac{p}{2}=q.$ Clearly,
$$\|\left(\sum_{j=1}^{\infty}
|T_{\alpha}(\lambda_j)f_j|^2\right)^{\frac{1}{2}}\|_p
= \|\sum_{j=1}^\infty |T_{\alpha}(\lambda_j)f_j|^2\|_q^{\frac{1}{2}}.$$
So, it is sufficient to deal with
$\|\sum_{j=1}^\infty |T_{\alpha}(\lambda_j)f_j|^2\|_q.$
As we know
$$ \|\sum_{j=1}^\infty |T_{\alpha}(\lambda_j)f_j|^2\|_q=
\sup_{\|\varphi\|_{q'}\leq 1,\varphi\in C_c^{\infty}(\R^n)}\bigg|
\int_{\R^n} \sum_{j=1}^\infty |T_{\alpha}(\lambda_j)f_j(x)|^2\varphi(x)dx\bigg| $$
it is enough to estimate the integral on the right hand side. In view
of Proposition 3.1 for $\Re{(\alpha)}>\frac{n-1}{2},$ we have
$$ \int_{\R^n} |T_{\alpha}(\lambda_j)f_j(x)|^2|\varphi(x)| dx \leq C
\int_{\R^n} |f_j(x)|^2 M\varphi(x)~~ dx,$$
where $C$ is independent of $\varphi$ and $\lambda_j.$ Therefore,
$$ \bigg|\int_{\R^n} \sum_{j=1}^\infty |T_{\alpha}(\lambda_j)f_j(x)|^2
\varphi(x) dx \bigg| \leq \sum_{j=1}^\infty \int_{\R^n}
|T_{\alpha}(\lambda_j)f_j(x)|^2 |\varphi(x)| dx $$
$$\leq C \int_{\R^n} \sum_{j=1}^\infty |f_j(x)|^2 M\varphi(x) dx.$$
By applying Holder's inequality to the right hand side of the above  we get
$$\bigg|\int_{\R^n} \sum_{j=1}^\infty |T_{\alpha}(\lambda_j)f_j(x)|^2
\varphi(x) dx\bigg|
\leq C \|\sum_{j=1}^\infty |f_j|^2\|_q \|M\varphi\|_{q'},$$
for $\Re{(\alpha)}>\frac{n-1}{2}.$
Since $q'>1$ and $\|\varphi\|_{q'}\leq 1,$ by  the boundedness of the
Hardy-Littlewood maximal function on $L^{q'}(\R^n),$ we get
$$\sup_{\|\varphi\|_{q'}\leq1,\varphi\in C_c^{\infty}(\R^n)}
\bigg|\int_{\R^n} \sum_{j=1}^\infty |T_{\alpha}(\lambda_j)f_j(x)|^2
\varphi(x)~~dx \bigg| \leq C \|\sum_{j=1}^\infty |f_j|^2\|_q, $$ for
$\Re{(\alpha)}>\frac{n-1}{2}.$ Hence we get the required vector-valued
inequality for $T_{\alpha}(\lambda).$

\end{proof}

\subsection{The R-boundedness of $\lambda \frac{d}{d\lambda}m_{\alpha}
(H(\lambda))$}

In this subsection we prove the vector valued inequality required to
establish the R-boundedness of the family
$ \lambda\frac{d}{d\lambda} m_{\alpha}(H(\lambda)).$ Without loss of
generality we assume that $\lambda > 0$ as the case of $\lambda<0$ follows in
a similar fashion with $\lambda$ replaced by $-\lambda$. The derivative of
$ m_\alpha(H(\lambda)) $ has been calculated in our earlier work \cite{JST},
see Lemma 3.4. It has been shown that $\lambda\frac{d}{d\lambda}
m_\alpha(H(\lambda))$ is a linear combination of terms of the form
$$ A_j^2(\lambda)\int_{0}^1{m'}_\alpha(H(\lambda)+2s\lambda) ds ,~~~~
A_j^{*2}(\lambda)\int_{0}^1{m'}_\alpha(H(\lambda)+2s\lambda )ds $$
and $ H(\lambda){m'}_\alpha(H(\lambda)).$

\begin{thm} The family $ S_\alpha(\lambda) =
\lambda\frac{d}{d\lambda}T_{\alpha}(\lambda)$ satisfies the inequality
$$ \|\left(\sum_{j=1}^{\infty}|S_\alpha(\lambda_j)f_j|^2\right)^{\frac{1}{2}}
\|_p \leq C \|\left(\sum_{j=1}^{\infty}|f_j|^2\right)^{\frac{1}{2}}\|_p $$
for  $1<p<\infty $ when $\Re{\alpha}>\frac{n+1}{2} $ for all choices of
$ \lambda_j \in \R^* $ and $ f_j \in L^p(\R^n).$
\end{thm}

As we have noted above, $ S_\alpha(\lambda) $ is a linear combination of
several terms. We will show that each term satisfies the above vector valued
inequality. Since we have already taken care of $ m_\alpha(H(\lambda))$ we
will begin with the term $ H(\lambda){m'}_\alpha(H(\lambda)).$ Recalling that
$ m_\alpha(u) = \frac{J_\alpha(\sqrt{u})}{\sqrt{u}^\alpha}, $ in view of
the relation
$\frac{d}{dt}\frac{J_{\alpha}(\sqrt{t})}{\sqrt{t}^{\alpha}}=
\frac{-1}{2}\frac{J_{\alpha+1}(\sqrt{t})}{\sqrt{t}^{\alpha+1}}$
we get
$$  H(\lambda){m'}_\alpha(H(\lambda))=
-\frac{1}{2}H(\lambda)m_{\alpha+1}(H(\lambda)).$$
The required maximal theorem for this family has been proved at the end of
the previous section. The R-boundedness of this family can now be proved
repeating the proofs of Proposition 3.1 and Theorem 3.2.

We will now sketch the proof of the vector valued inequality for
the remaining terms. We will only consider the term
$$ A_j^2(\lambda)\int_{0}^1 {m'}_\alpha(H(\lambda)+2s\lambda)  ds,$$
as the other one can be treated similarly. As observed above
$$ {m'}_\alpha(H(\lambda)+2s\lambda) = -\frac{1}{2} m_{\alpha+1}
(H(\lambda)+2s\lambda)$$ and hence we have to consider
$$ A_j^2(\lambda)H(\lambda)^{-1} \int_{0}^1 H(\lambda) m_{\alpha+1}
(H(\lambda)+2s\lambda) ds .$$
As was shown in \cite{JST} the operator $ A_j^2(\lambda)H(\lambda)^{-1} $
turns out to be a Calderon-Zygmund singular integral operator whose CZ
constants are uniform in $ \lambda.$ Hence by a theorem of Cordoba and
Fefferman \cite{CF} the family $ A_j^2(\lambda)H(\lambda)^{-1} $ satisfies
a vector valued inequality. (See Theorem 2.1 in \cite{JST})

Finally we are left with the operator family
$$ \int_{0}^1 H(\lambda) m_{\alpha+1}(H(\lambda)+2s\lambda) ds $$
and hence it is enough to show that the family
$$ H(\lambda) m_{\alpha+1}(H(\lambda)+2s\lambda) $$
is R-bounded uniformly in $ s \in (0,1).$ But the treatment of this is very
similar to that of $  H(\lambda) m_{\alpha+1}(H(\lambda)) $ which we
considered before using the maximal Theorem 2.3. Once again we leave the
details to the reader.
This completes the proof.


\begin{thebibliography}{99}

\bibitem{FS} C. Fefferman and E. M. Stein, \textit{Maximal Inequalities},
American Journal of Mathematics {93, no.1} (1971), 107-115.

\bibitem{CF} A. Cordoba and C. Fefferman, \textit{A weighted norm inequality for singular integral operators}
Studia Mathematica (1976), 97-101.


\bibitem{JST} K. Jotsaroop, P. K. Sanjay and S. Thangavelu,
 \textit{Riesz Transform and multipliers for Grushin Operator},
J. Analyse Math. (to appear).

\bibitem{RM} R. Meyer, \textit{$L^p $ estimates for the Grushin operator},
arXiv:0709.2188 (2007).


\bibitem{AM} A. Miyachi, \textit{On some estimates for the wave equation in $L^p$ and $H^p,$}
J. Fac. Sci. Univ. Tokyo Sect. IA Math.{ 27}(1980), 331-354.


\bibitem{MS} D. Mueller and E. M. Stein, \textit{$L^p$ estimates for the wave equation on the Heisenberg group},
Revista Matematica Iberoamericana, { 15,no. 2}(1999), 297-334.

\bibitem{NT} E. K. Narayanan and S. Thangavelu, \textit{Oscillating Multipliers for some Eigenfunction Expansions},
Journal of Fourier Analysis and Applications,  Vol. 7, Issue 4, (2001), 
373-394.

\bibitem{NS}E. K. Narayanan and S. Thangavelu, \textit{Oscillating Multipliers on the Heisenberg Group,}
Colloquium Mathematicum, Vol. 90, Issue no.1, (2001), 37-50.

\bibitem{JP} J. Peral, \textit{$L^p$ estimates for the wave equation},
Journal of Functional Analysis, {36}(1980), 114-145.


\bibitem{ES} E. M. Stein, \textit{Introduction to Fourier Analysis on Euclidian Spaces}, Princeton University Press (1971).

\bibitem{ES2} E. M. Stein, \textit{Singular Integrals and Differentiability Properties of Functions}, Princeton University Press (1970).

\bibitem{RS} R. Strichartz, \textit{Convolution with kernels having singularities on a sphere,} Transactions of the American Mathematical Society, Vol. 148, (1970), 461-471.

\bibitem{ST} S. Thangavelu, \textit{Lectures on Hermite and Laguerre
expansions}, Math. Notes No. 42, Princeton Univ. Press (1993).

\bibitem{GW}G. M. Watson, \textit{A Treatise on the theory of Bessel functions}, Cambridge University Press (1922).

\bibitem{LW} L. Weis, \textit{Operator valued Fourier multiplier theorems and
maximal $ L^p $ regularity}, Math. Ann. 319 (2001), 735-758.


\bibitem{JZ} J. Zhong, \textit{Harmonic Analysis for some Schrodinger Operators}, Princeton University thesis.









\end{thebibliography}
\end{document}